\documentclass{article}%
\usepackage{amsfonts}
\usepackage{graphicx}
\usepackage{amsmath}%
\usepackage{amssymb}
\providecommand{\U}[1]{\protect\rule{.1in}{.1in}}
\newtheorem{theorem}{Theorem}

\newtheorem{corollary}[theorem]{Corollary}

\newtheorem{definition}[theorem]{Definition}

\newtheorem{lemma}[theorem]{Lemma}
\newtheorem{notation}[theorem]{Notation}

\newtheorem{proposition}[theorem]{Proposition}
\newtheorem{remark}[theorem]{Remark}

\newenvironment{proof}[1][Proof]{\textbf{#1.} }{\ \rule{0.5em}{0.5em}}
\begin{document}

\title{Synthetic Differential Geometry\\of \\Chen's Iterated Integrals}
\author{Hirokazu NISHIMURA\\Institute of Mathematics, University of Tsukuba\\Tsukuba, Ibaraki, 305-8571\\JAPAN}
\maketitle

\begin{abstract}
Chen's iterated integrals are treated within synthetic differential geometry.
The main result is that iterated integrals produce a subcomplex of the de Rham
complex on the free path space as well as based path spaces..

\end{abstract}

\section{\label{s1}Introduction}

Chen's iterated integrals have been introduced and investigated in
\cite{chen1}, \cite{chen2}, \cite{chen3} and others. As far as we know, there
is only one volume on them, namely, \cite{koh}, though it is unfortunately
written not in English but in Japanese. Chen has established two fundamental
theorems on them, the first claiming that the cohomology of the bar complex
$\mathcal{B}^{\ast}\left(  M\right)  $\ generated by Chen's iterated integrals
on the loop space $\Omega M$ of a smooth manifold $M$ is no other than the
cohomology of the loop space $\Omega M$, so long as $M$ is simply connected,
while the second asserting that
\[
\mathcal{F}^{-k}H^{0}\left(  \mathcal{B}^{\ast}\left(  M\right)  \right)
\cong\mathrm{Hom}\left(  \mathbf{Z}\pi_{1}\left(  M,x_{0}\right)
/J^{k+1},\mathbf{R}\right)
\]
where $\mathcal{F}^{-k}$ is the filtration determined by Chen's iterated
integrals of length $k$ or less, $\mathbf{Z}\pi_{1}\left(  M,x_{0}\right)  $
denotes the group ring of the fundamental group $\pi_{1}\left(  M,x_{0}%
\right)  $, and $J$ stands for the kernel of the augmentation mapping
$\mathbf{Z}\pi_{1}\left(  M,x_{0}\right)  \rightarrow\mathbf{Z}$. Chen's
iterated integrals have been applied successfully to various branches of
mathematics, say, to Vassiliev invariants of knots and braids (\cite{koh'} and
\cite{kon}), algebraic cycles (\cite{ha}) and multiple zeta functions (
\cite{dri} and \cite{te}).

The principal objective in this paper is to establish that Chen's iterated
integrals yield a subcomplex of the de Rham complex on the free path space
$\mathcal{P}M$\ as well as the path space $\mathcal{P}_{x_{1}}M$\ with the
starting point based, the path space $\mathcal{P}^{x_{2}}M$\ with the
terminating point based \ and the path space $\mathcal{P}_{x_{1}}^{x_{2}}%
M$\ with both the starting and terminating points based within our favorite
framework of synthetic differential geometry (\cite{ko} and \cite{la}). Chen's
two fundamental theorems will be dealt with synthetically in subsequent papers.

Even if $M$\ is a finite-dimensional smooth manifold, the free path space
$\mathcal{P}M$\ as well as the path space $\mathcal{P}_{x_{1}}M$\ with the
starting point based, the path space $\mathcal{P}^{x_{2}}M$\ with the
terminating point based \ and the path space $\mathcal{P}_{x_{1}}^{x_{2}}%
M$\ with both the starting and terminating points based\ is generally
infinite-dimensional in essence, so that we must choose an adequate framework
for \textit{infinite-dimensional} differential geometry in order to deal with
these spaces. Chen has chosen so-called diffeology, in which plots enable one
to do constructions and computations with coordinates. For an excellent volume
on diffeology, one is referred to \cite{ig}. Since we have chosen synthetic
differential geometry and we assume $M$\ only to be a microlinear space, we
are coerced into doing everything in a coordinate-free way.

In synthetic differential geometry, one has to do everything within an
esoteric topos, which alienates many mathematicians. However we could
emancipate synthetic differential geometry from topos theory, resulting in
\textit{axiomatic differential geometry}, in which Weil functors play a
pivotal role. For the first steps in axiomatic differential geometry, one is
referred to \cite{ni1}, \cite{ni2} and \cite{ni3}. For an excellent
investigation on the relationship among a few infinite-dimensional
differential geometries (including diffeology) from a standpoint of category
theory, one is referred to \cite{st}.

\section{\label{s2}Preliminaries}

The reader is referred to \cite{ko} and \cite{la} for synthetic differential
geometry. In particular, the reader is assumed to be familiar with the first
four chapters of \cite{la}.

\begin{notation}
We denote by $M$ an arbitrary microlinear space.
\end{notation}

\begin{notation}
We denote by $I$ the unit interval $\left[  0,1\right]  $.
\end{notation}

\begin{remark}
As is discussed in \S 3.2 of \cite{la}, vector fields on $M$\ can be viewed
from three related but distinct standpoints. The first is to see them
orthodoxically as mappings $M\rightarrow M^{D}$ (sections of tangent bundles),
the second is to put them down at mappings $D\times M\rightarrow M$
(infinitesimal flow), and the third is, most radically, to regard them as
mappings $D\rightarrow M^{M}$ (infinitesimal transformation), though we prefer
the third viewpoint most.
\end{remark}

\begin{notation}
We denote by $\mathcal{A}^{p}\left(  M\right)  $ the totality of differential
forms on $M$\ of degree $p$, $\mathcal{A}\left(  M\right)  $ designating the
totality of differential forms on $M$.
\end{notation}

\begin{notation}
We denote by $\mathbf{d}$\ the exterior differentiation. Given a vector field
$X$ on $M$, we denote by $\mathbf{i}_{X}$ and $\mathbf{L}_{X}$ the interior
product and the Lie derivative with respect to the vector field $X$\ respectively.
\end{notation}

\begin{remark}
We have a natural pairing
\[
\left\langle \left(  \gamma;d_{1},...,d_{p}\right)  ,\omega\right\rangle
=d_{1}...d_{p}\omega\left(  \gamma\right)
\]
for any $\left(  \gamma;d_{1},...,d_{p}\right)  \in M^{D^{p}}\times D^{p}%
$\ and any $\omega\in\mathcal{A}^{p}\left(  M\right)  $. Indeed, differential
forms can be characterized as mappings on $M^{D^{p}}\times D^{p}$ abiding by
certain properties, for which the reader is referred to Proposition 2 in
\S 4.2 of \cite{la}.
\end{remark}

The following is one of the three Cartan formulas for differential forms and
will be used in our discussions.

\begin{theorem}
\label{t2.2}Given a vector field $X$ on $M$, we have
\[
\mathbf{L}_{X}=\mathbf{di}_{X}+\mathbf{di}_{X}%
\]

\end{theorem}

\section{\label{s3}Simple Integrals}

\begin{notation}
We denote by $\mathcal{P}M$\ the set
\[
\mathcal{P}M=M^{I}%
\]

\end{notation}

\begin{notation}
We denote by $\varphi$\ the mapping $\varphi:I\times\mathcal{P}M\rightarrow M$
defined by
\[
\varphi\left(  t,\theta\right)  =\theta\left(  t\right)
\]
for any $\left(  t,\theta\right)  \in I\times\mathcal{P}M$.
\end{notation}

\begin{notation}
Given $t\in I$, we denote by $\iota_{t}:\mathcal{P}M\rightarrow$
$I\times\mathcal{P}M$ the mapping
\[
\theta\in\mathcal{P}M\mapsto\left(  t,\theta\right)  \in I\times\mathcal{P}M
\]

\end{notation}

\begin{notation}
Given $t\in I$, we denote by $\varphi_{t}$ the mapping
\[
\varphi\circ\iota_{t}:\mathcal{P}M\rightarrow M
\]

\end{notation}

\begin{notation}
We denote by $\frac{\partial}{\partial t}$ the vector field
\[
\left(  d,\left(  t,\theta\right)  \right)  \in D\times\left(  I\times
\mathcal{P}M\right)  \mapsto\left(  t+d,\theta\right)  \in I\times\mathcal{P}M
\]
on $I\times\mathcal{P}M$.
\end{notation}

\begin{notation}
Given $\omega\in\mathcal{A}^{p}\left(  M\right)  $\ with $p$\ being a positive
integer and $t\in I$, we write
\[
\left(  \omega\right)  _{t}^{\ast}\in\mathcal{A}^{p-1}\left(  \mathcal{P}%
M\right)
\]
for
\[
\iota_{t}^{\ast}\mathbf{i}_{\frac{\partial}{\partial t}}\varphi^{\ast}\omega
\]

\end{notation}

\begin{proposition}
\label{t3.1}Given $\omega_{1}\in\mathcal{A}^{p_{1}}\left(  M\right)  $ and
$\omega_{2}\in\mathcal{A}^{p_{2}}\left(  M\right)  $ with $p_{1}$ and $p_{2}$
being positive integers, we have
\[
\left(  \omega_{1}\wedge\omega_{2}\right)  _{t}^{\ast}=\left(  \omega
_{1}\right)  _{t}^{\ast}\wedge\varphi_{t}^{\ast}\omega_{2}+\left(  -1\right)
^{p_{1}}\varphi_{t}^{\ast}\omega_{1}\wedge\left(  \omega_{2}\right)
_{t}^{\ast}%
\]
for any $t\in I$.
\end{proposition}

\begin{proof}
We have
\begin{align*}
& \left(  \omega_{1}\wedge\omega_{2}\right)  _{t}^{\ast}\\
& =\iota_{t}^{\ast}\mathbf{i}_{\frac{\partial}{\partial t}}\varphi^{\ast
}\left(  \omega_{1}\wedge\omega_{2}\right) \\
& =\iota_{t}^{\ast}\mathbf{i}_{\frac{\partial}{\partial t}}\left(
\varphi^{\ast}\omega_{1}\wedge\varphi^{\ast}\omega_{2}\right) \\
& =\iota_{t}^{\ast}\left(  \mathbf{i}_{\frac{\partial}{\partial t}}%
\varphi^{\ast}\omega_{1}\wedge\varphi^{\ast}\omega_{2}+\left(  -1\right)
^{p_{1}}\varphi^{\ast}\omega_{1}\wedge\mathbf{i}_{\frac{\partial}{\partial t}%
}\varphi^{\ast}\omega_{2}\right) \\
& =\iota_{t}^{\ast}\mathbf{i}_{\frac{\partial}{\partial t}}\varphi^{\ast
}\omega_{1}\wedge\iota_{t}^{\ast}\varphi^{\ast}\omega_{2}+\left(  -1\right)
^{p_{1}}\iota_{t}^{\ast}\varphi^{\ast}\omega_{1}\wedge\iota_{t}^{\ast
}\mathbf{i}_{\frac{\partial}{\partial t}}\varphi^{\ast}\omega_{2}\\
& =\left(  \omega_{1}\right)  _{t}^{\ast}\wedge\varphi_{t}^{\ast}\omega
_{2}+\left(  -1\right)  ^{p_{1}}\varphi_{t}^{\ast}\omega_{1}\wedge\left(
\omega_{2}\right)  _{t}^{\ast}%
\end{align*}

\end{proof}

\begin{definition}
Given a mapping $\widetilde{\omega}:I\rightarrow\mathcal{A}^{p}\left(
\mathcal{P}M\right)  $\ with $p$\ being a natural number and $s,t\in I$, we
define
\[
\int_{s}^{t}\widetilde{\omega}\left(  u\right)  \mathbf{d}u\in\mathcal{A}%
^{p}\left(  \mathcal{P}M\right)
\]
to be such that
\begin{align*}
& \left\langle \left(  \gamma;d_{1},...,d_{p}\right)  ,\int_{s}^{t}%
\widetilde{\omega}\left(  u\right)  \mathbf{d}u\right\rangle \\
& =\int_{s}^{t}\left\langle \left(  \gamma;d_{1},...,d_{p}\right)
,\widetilde{\omega}\left(  u\right)  \right\rangle \mathbf{d}u
\end{align*}
for any $\left(  \gamma;d_{1},...,d_{p}\right)  \in\left(  \mathcal{P}%
M\right)  ^{D^{p}}\times D^{p}$.
\end{definition}

It is easy to see that

\begin{proposition}
\label{t3.4}Given a mapping $\widetilde{\omega}:I\rightarrow\mathcal{A}%
^{p}\left(  \mathcal{P}M\right)  $\ with $p$\ being a natural number and
$s,t\in I$, we have
\[
\mathbf{d}\int_{s}^{t}\widetilde{\omega}\left(  u\right)  \mathbf{d}u=\int
_{s}^{t}\mathbf{d}\widetilde{\omega}\left(  u\right)  \mathbf{d}u
\]

\end{proposition}

Now we are ready to give a definition of Chen's single integral, which is the
starting point of his iterated integrals.

\begin{definition}
Given $\omega\in\mathcal{A}^{p}\left(  M\right)  $\ with $p$\ being a positive
integer and $s,t\in I$, we define
\[
\int_{s}^{t}\omega\in\mathcal{A}^{p-1}\left(  \mathcal{P}M\right)
\]
to be
\[
\int_{s}^{t}\left(  \omega\right)  _{u}^{\ast}\mathbf{d}u
\]

\end{definition}

\begin{notation}
Given $\omega\in\mathcal{A}^{p}\left(  M\right)  $\ with $p$\ being a positive
integer, we write
\[
\int\omega\in\mathcal{A}^{p-1}\left(  \mathcal{P}M\right)
\]
for
\[
\int_{0}^{1}\omega
\]

\end{notation}

\begin{proposition}
\label{t3.2}Given $\omega\in\mathcal{A}^{p}\left(  M\right)  $\ with
$p$\ being a positive integer and $t\in I$, we have
\[
\mathbf{d}\left(  \omega\right)  _{t}^{\ast}=\iota_{t}^{\ast}\mathbf{L}%
_{\frac{\partial}{\partial t}}\varphi^{\ast}\omega-\left(  \mathbf{d}%
\omega\right)  _{t}^{\ast}%
\]

\end{proposition}

\begin{proof}
We have
\begin{align*}
& \mathbf{d}\left(  \omega\right)  _{t}^{\ast}\\
& =\mathbf{d}\iota_{t}^{\ast}\mathbf{i}_{\frac{\partial}{\partial t}}%
\varphi^{\ast}\omega\\
& =\iota_{t}^{\ast}\mathbf{di}_{\frac{\partial}{\partial t}}\varphi^{\ast
}\omega\\
& =\iota_{t}^{\ast}\left(  \mathbf{L}_{\frac{\partial}{\partial t}}%
-\mathbf{i}_{\frac{\partial}{\partial t}}\mathbf{d}\right)  \varphi^{\ast
}\omega\\
& \text{[By Theorem \ref{t2.2}]}\\
& =\iota_{t}^{\ast}\mathbf{L}_{\frac{\partial}{\partial t}}\varphi^{\ast
}\omega-\iota_{t}^{\ast}\mathbf{i}_{\frac{\partial}{\partial t}}%
\mathbf{d}\varphi^{\ast}\omega\\
& =\iota_{t}^{\ast}\mathbf{L}_{\frac{\partial}{\partial t}}\varphi^{\ast
}\omega-\left(  \mathbf{d}\omega\right)  _{t}^{\ast}%
\end{align*}

\end{proof}

\begin{corollary}
\label{t3.2.1}Let $t\in I$ and $d\in D$. Given $\omega\in\mathcal{A}%
^{p}\left(  M\right)  $\ with $p$\ being a positive integer, we have
\[
\mathbf{d}\int_{t}^{t+d}\omega=-\int_{t}^{t+d}\mathbf{d}\omega-\varphi
_{t}^{\ast}\omega+\varphi_{t+d}^{\ast}\omega
\]

\end{corollary}

\begin{proof}
It suffices to show that
\begin{align*}
& \left\langle \left(  \gamma;d_{1},...,d_{p}\right)  ,\mathbf{d}\int
_{t}^{t+d}\omega\right\rangle \\
& =\left\langle \left(  \gamma;d_{1},...,d_{p}\right)  ,-\int_{t}%
^{t+d}\mathbf{d}\omega-\varphi_{t}^{\ast}\omega+\varphi_{t+d}^{\ast}%
\omega\right\rangle
\end{align*}
for any $\left(  \gamma;d_{1},...,d_{p}\right)  \in M^{D^{p}}\times D^{p}$,
which follows from the following computation:
\begin{align*}
& \left\langle \left(  \gamma;d_{1},...,d_{p}\right)  ,\mathbf{d}\int
_{t}^{t+d}\omega\right\rangle \\
& =\left\langle \partial\left(  \gamma;d_{1},...,d_{p}\right)  ,\int_{t}%
^{t+d}\omega\right\rangle \\
& =\int_{t}^{t+d}\left\langle \partial\left(  \gamma;d_{1},...,d_{p}\right)
,\left(  \omega\right)  _{u}^{\ast}\right\rangle \mathbf{d}u\\
& =d\left\langle \partial\left(  \gamma;d_{1},...,d_{p}\right)  ,\left(
\omega\right)  _{t}^{\ast}\right\rangle \\
& =d\left\langle \left(  \gamma;d_{1},...,d_{p}\right)  ,\mathbf{d}\left(
\omega\right)  _{t}^{\ast}\right\rangle \\
& =d\left\langle \left(  \gamma;d_{1},...,d_{p}\right)  ,\iota_{t}^{\ast
}\mathbf{L}_{\frac{\partial}{\partial t}}\varphi^{\ast}\omega-\left(
\mathbf{d}\omega\right)  _{t}^{\ast}\right\rangle \\
& \text{\lbrack By Proposition \ref{t3.2}]}\\
& =\left\langle \left(  \gamma;d_{1},...,d_{p}\right)  ,d\iota_{t}^{\ast
}\mathbf{L}_{\frac{\partial}{\partial t}}\varphi^{\ast}\omega\right\rangle
-d\left\langle \left(  \gamma;d_{1},...,d_{p}\right)  ,\left(  \mathbf{d}%
\omega\right)  _{t}^{\ast}\right\rangle \\
& =\left\langle \left(  \gamma;d_{1},...,d_{p}\right)  ,\varphi_{t+d}^{\ast
}\omega-\varphi_{t}^{\ast}\omega\right\rangle -\left\langle \left(
\gamma;d_{1},...,d_{p}\right)  ,\int_{t}^{t+d}\mathbf{d}\omega\right\rangle \\
& =\left\langle \left(  \gamma;d_{1},...,d_{p}\right)  ,-\int_{t}%
^{t+d}\mathbf{d}\omega-\varphi_{t}^{\ast}\omega+\varphi_{t+d}^{\ast}%
\omega\right\rangle
\end{align*}

\end{proof}

\begin{corollary}
\label{t3.2.2}Given $\omega\in\mathcal{A}^{p}\left(  M\right)  $\ with
$p$\ being a positive integer, we have
\[
\mathbf{d}\int_{s}^{t}\omega=-\int_{s}^{t}\mathbf{d}\omega-\varphi_{s}^{\ast
}\omega+\varphi_{t}^{\ast}\omega
\]
for any $s,t\in I$. In particular,
\[
\mathbf{d}\int\omega=-\int\mathbf{d}\omega-\varphi_{0}^{\ast}\omega
+\varphi_{1}^{\ast}\omega
\]

\end{corollary}

\begin{proof}
Let us define a function $F:I\rightarrow\mathbb{R}$ to be
\[
F(u)=\left\langle \left(  \gamma;d_{1},...,d_{p}\right)  ,\mathbf{d}\int
_{s}^{u}\omega+\int_{s}^{u}\mathbf{d}\omega+\varphi_{s}^{\ast}\omega
-\varphi_{u}^{\ast}\omega\right\rangle
\]
for any $u\in I$. Then we have
\begin{align*}
& F(u+d)-F(u)\\
& =\left\langle \left(  \gamma;d_{1},...,d_{p}\right)  ,\mathbf{d}\int
_{u}^{u+d}\omega+\int_{u}^{u+d}\mathbf{d}\omega+\varphi_{u}^{\ast}%
\omega-\varphi_{u+d}^{\ast}\omega\right\rangle \\
& =0
\end{align*}
by dint of the above corollary, which implies that
\[
F^{\prime}(u)=0
\]
for any $u\in I$. Since
\[
F\left(  s\right)  =0
\]
holds trivially, we are done.
\end{proof}

It is easy to see that

\begin{proposition}
\label{t3.3}Given $\omega\in\mathcal{A}^{p}\left(  M\right)  $\ with
$p$\ being a positive integer and $s,s^{\prime},t,t^{\prime}\in I$, we have
\[
\int_{s}^{t}\omega=\int_{s}^{s^{\prime}}\omega+\int_{s^{\prime}}^{t^{\prime}%
}\omega+\int_{t^{\prime}}^{t}\omega
\]

\end{proposition}

\section{\label{s4}Iterated Integrals}

\begin{definition}
Given $s_{1},...,s_{k},t\in I$ and $\omega_{1}\in\mathcal{A}^{p_{1}}\left(
M\right)  ,...,\omega_{k}\in\mathcal{A}^{p_{k}}\left(  M\right)  $ with
$p_{1},...,p_{k}$\ being positive integers, we define
\[
\int_{s_{1},...,s_{k}}^{t}\omega_{1}...\omega_{k}\in\mathcal{A}^{p_{1}%
+...+p_{k}-k}\left(  \mathcal{P}M\right)
\]
by induction on $k$\ to be
\[
\int_{s_{k}}^{t}\left(  \left(  \int_{s_{1},...,s_{k-1}}^{u}\omega
_{1}...\omega_{k-1}\right)  \wedge\left(  \omega_{k}\right)  _{u}^{\ast
}\right)  \mathbf{d}u
\]
By way of example, we have
\[
\int_{s_{1},s_{2}}^{t}\omega_{1}\omega_{2}=\int_{s_{2}}^{t}\left(  \left(
\int_{s_{1}}^{u}\omega_{1}\right)  \wedge\left(  \omega_{2}\right)  _{u}%
^{\ast}\right)  \mathbf{d}u
\]

\end{definition}

\begin{notation}
Given $\omega_{1}\in\mathcal{A}^{p_{1}}\left(  M\right)  ,...,\omega_{k}%
\in\mathcal{A}^{p_{k}}\left(  M\right)  $ with $p_{1},...,p_{k}$\ being
positive integers, we write
\[
\int\omega_{1}...\omega_{k}%
\]
for
\[
\int_{0,...,0}^{1}\omega_{1}...\omega_{k}%
\]

\end{notation}

\begin{notation}
Since the space $\mathcal{A}^{p}\left(  \mathcal{P}M\right)  $ with $p$\ being
a natural number is a Euclidean $\mathbb{R}$-module, any mapping
$\widetilde{\omega}:I\rightarrow\mathcal{A}^{p}\left(  \mathcal{P}M\right)  $
and any $t\in I$ give rise to a unique $\mathbf{D}_{t}\widetilde{\omega}%
\in\mathcal{A}^{p}\left(  \mathcal{P}M\right)  $ such that
\[
\widetilde{\omega}\left(  t+d\right)  -\widetilde{\omega}\left(  t\right)
=d\mathbf{D}_{t}\widetilde{\omega}%
\]
for any $d\in D$.
\end{notation}

It is easy to see that

\begin{proposition}
\label{t4.1}Let us suppose that we are given $\omega\in\mathcal{A}^{p}\left(
M\right)  $\ with $p$\ being a positive integer and $t\in I$. Let
$\widetilde{\omega}:I\rightarrow\mathcal{A}^{p}\left(  \mathcal{P}M\right)  $
be the mapping $s\in I\mapsto\varphi_{s}^{\ast}\omega$. Then we have
\[
\mathbf{D}_{t}\left(  \widetilde{\omega}\right)  =\iota_{t}^{\ast}%
\mathbf{L}_{\frac{\partial}{\partial t}}\varphi^{\ast}\omega
\]

\end{proposition}

The following two are no other than variants of the fundametal theorem in calculus.

\begin{proposition}
\label{t4.2}Given $s,t\in I$ and a mapping $\widetilde{\omega}:I\rightarrow
\mathcal{A}^{p}\left(  \mathcal{P}M\right)  $ with $p$\ being a natural
number, we have
\[
\mathbf{D}_{t}\left(  \int_{s}^{t}\widetilde{\omega}\left(  u\right)
\mathbf{d}u\right)  =\widetilde{\omega}\left(  t\right)
\]

\end{proposition}

\begin{proposition}
\label{t4.3}Given $s,t\in I$ and a mapping $\widetilde{\omega}:I\rightarrow
\mathcal{A}^{p}\left(  \mathcal{P}M\right)  $ with $p$\ being a natural
number, we have
\[
\int_{s}^{t}\mathbf{D}_{u}\widetilde{\omega}\left(  u\right)  \mathbf{d}%
u=\widetilde{\omega}\left(  t\right)  -\widetilde{\omega}\left(  s\right)
\]

\end{proposition}

It is easy to see that

\begin{proposition}
\label{t4.4}Given mappings $\widetilde{\omega}_{1}:I\rightarrow\mathcal{A}%
^{p_{1}}\left(  \mathcal{P}M\right)  $\ and $\widetilde{\omega}_{2}%
:I\rightarrow\mathcal{A}^{p_{2}}\left(  \mathcal{P}M\right)  $\ with $p_{1}%
$\ and $p_{2}$ being natural numbers, we have
\[
\mathbf{D}_{t}\left(  \widetilde{\omega}_{1}\left(  t\right)  \wedge
\widetilde{\omega}_{2}\left(  t\right)  \right)  =\mathbf{D}_{t}%
\widetilde{\omega}_{1}\left(  t\right)  \wedge\widetilde{\omega}_{2}\left(
t\right)  +\widetilde{\omega}_{1}\left(  t\right)  \wedge\mathbf{D}%
_{t}\widetilde{\omega}_{2}\left(  t\right)
\]
for any $t\in I$.
\end{proposition}

\begin{lemma}
\label{t4.5}Let $s_{1},s_{2},t\in I$. Given $\omega_{1}\in\mathcal{A}^{p_{1}%
}\left(  M\right)  $ and $\omega_{2}\in\mathcal{A}^{p_{2}}\left(  M\right)  $
with $p_{1}$ and $p_{2}$ being positive integers, we have
\begin{align*}
& \int_{s_{2}}^{t}\left(  \omega_{1}\right)  _{u}^{\ast}\wedge\varphi
_{u}^{\ast}\omega_{2}\mathbf{d}u+\int_{s_{2}}^{t}\left(  \int_{s_{1}}%
^{u}\omega_{1}\right)  \wedge\left(  \iota_{u}^{\ast}\mathbf{L}_{\frac
{\partial}{\partial t}}\varphi^{\ast}\omega_{2}\right)  \mathbf{d}u\\
& =\left(  \int_{s_{1}}^{t}\omega_{1}\right)  \wedge\varphi_{t}^{\ast}%
\omega_{2}-\left(  \int_{s_{1}}^{s_{2}}\omega_{1}\right)  \wedge\varphi
_{s_{2}}^{\ast}\omega_{2}%
\end{align*}

\end{lemma}

\begin{proof}
We have
\[
\mathbf{D}_{u}\left(  \left(  \int_{s_{1}}^{u}\omega_{1}\right)  \wedge
\varphi_{u}^{\ast}\omega_{2}\right)  =\left(  \omega_{1}\right)  _{u}^{\ast
}\wedge\varphi_{u}^{\ast}\omega_{2}+\left(  \int_{s_{1}}^{u}\omega_{1}\right)
\wedge\left(  \iota_{u}^{\ast}\mathbf{L}_{\frac{\partial}{\partial t}}%
\varphi^{\ast}\omega_{2}\right)
\]
so that the desired formula follows by dint of Proposition \ref{t4.3}.
\end{proof}

\begin{theorem}
\label{t4.6}Let $s_{1},s_{2},t\in I$. Given $\omega_{1}\in\mathcal{A}^{p_{1}%
}\left(  M\right)  $ and $\omega_{2}\in\mathcal{A}^{p_{2}}\left(  M\right)  $
with $p_{1}$ and $p_{2}$ being positive integers, we have
\begin{align*}
& \mathbf{d}\int_{s_{1},s_{2}}^{t}\omega_{1}\omega_{2}\\
& =-\int_{s_{1},s_{2}}^{t}\left(  \mathbf{d}\omega_{1}\right)  \omega
_{2}+\left(  -1\right)  ^{p_{1}}\int_{s_{1},s_{2}}^{t}\omega_{1}\left(
\mathbf{d}\omega_{2}\right)  +\left(  -1\right)  ^{p_{1}}\int_{s_{2}}%
^{t}\omega_{1}\wedge\omega_{2}\\
& -\varphi_{s_{1}}^{\ast}\omega_{1}\wedge\int_{s_{2}}^{t}\omega_{2}+\left(
-1\right)  ^{p_{1}}\left(  \int_{s_{1}}^{s_{2}}\omega_{1}\right)
\wedge\varphi_{s_{2}}^{\ast}\omega_{2}-\left(  -1\right)  ^{p_{1}}\left(
\int_{s_{1}}^{t}\omega_{1}\right)  \wedge\varphi_{t}^{\ast}\omega_{2}%
\end{align*}
In particular, we have
\begin{align*}
& \mathbf{d}\int\omega_{1}\omega_{2}\\
& =-\int\left(  \mathbf{d}\omega_{1}\right)  \omega_{2}+\left(  -1\right)
^{p_{1}}\int\omega_{1}\left(  \mathbf{d}\omega_{2}\right)  +\left(  -1\right)
^{p_{1}}\int\omega_{1}\wedge\omega_{2}\\
& -\varphi_{0}^{\ast}\omega_{1}\wedge\int\omega_{2}-\left(  -1\right)
^{p_{1}}\left(  \int\omega_{1}\right)  \wedge\varphi_{1}^{\ast}\omega_{2}%
\end{align*}

\end{theorem}

\begin{proof}
We have
\begin{align*}
& \mathbf{d}\int_{s_{1},s_{2}}^{t}\omega_{1}\omega_{2}\\
& =\mathbf{d}\int_{s_{2}}^{t}\left(  \int_{s_{1}}^{u}\omega_{1}\right)
\wedge\left(  \omega_{2}\right)  _{u}^{\ast}\mathbf{d}u\\
& =\int_{s_{2}}^{t}\left\{  \left(  \mathbf{d}\int_{s_{1}}^{u}\omega
_{1}\right)  \wedge\left(  \omega_{2}\right)  _{u}^{\ast}+\left(  -1\right)
^{p_{1}-1}\left(  \int_{s_{1}}^{u}\omega_{1}\right)  \wedge\mathbf{d}\left(
\omega_{2}\right)  _{u}^{\ast}\right\}  \mathbf{d}u\\
& \text{[By Proposition \ref{t3.4}]}\\
& =\int_{s_{2}}^{t}\left(  \mathbf{d}\int_{s_{1}}^{u}\omega_{1}\right)
\wedge\left(  \omega_{2}\right)  _{u}^{\ast}\mathbf{d}u+\left(  -1\right)
^{p_{1}-1}\int_{s_{2}}^{t}\left(  \int_{s_{1}}^{u}\omega_{1}\right)
\wedge\mathbf{d}\left(  \omega_{2}\right)  _{u}^{\ast}\mathbf{d}u\\
& =\int_{s_{2}}^{t}\left(  -\int_{s_{1}}^{u}\mathbf{d}\omega_{1}%
-\varphi_{s_{1}}^{\ast}\omega_{1}+\varphi_{u}^{\ast}\omega_{1}\right)
\wedge\left(  \omega_{2}\right)  _{u}^{\ast}\mathbf{d}u\\
& +\left(  -1\right)  ^{p_{1}-1}\int_{s_{2}}^{t}\left(  \int_{s_{1}}^{u}%
\omega_{1}\right)  \wedge\left(  \iota_{u}^{\ast}\mathbf{L}_{\frac{\partial
}{\partial t}}\varphi^{\ast}\omega_{2}-\left(  \mathbf{d}\omega_{2}\right)
_{u}^{\ast}\right)  \mathbf{d}u\\
& \text{[By Proposition \ref{t3.2} and Corollary \ref{t3.2.2}]}\\
& =-\int_{s_{1},s_{2}}^{t}\left(  \mathbf{d}\omega_{1}\right)  \omega
_{2}+\left(  -1\right)  ^{p_{1}}\int_{s_{1},s_{2}}^{t}\omega_{1}\left(
\mathbf{d}\omega_{2}\right)  -\left(  \varphi_{s_{1}}^{\ast}\omega_{1}\right)
\wedge\left(  \int_{s_{2}}^{t}\left(  \omega_{2}\right)  _{u}^{\ast}%
\mathbf{d}u\right) \\
& +\int_{s_{2}}^{t}\left(  \varphi_{u}^{\ast}\omega_{1}\right)  \wedge\left(
\omega_{2}\right)  _{u}^{\ast}\mathbf{d}u-\left(  -1\right)  ^{p_{1}}%
\int_{s_{2}}^{t}\left(  \int_{s_{1}}^{u}\omega_{1}\right)  \wedge\left(
\iota_{u}^{\ast}\mathbf{L}_{\frac{\partial}{\partial t}}\varphi^{\ast}%
\omega_{2}\right)  \mathbf{d}u\\
& =-\int_{s_{1},s_{2}}^{t}\left(  \mathbf{d}\omega_{1}\right)  \omega
_{2}+\left(  -1\right)  ^{p_{1}}\int_{s_{1},s_{2}}^{t}\omega_{1}\left(
\mathbf{d}\omega_{2}\right) \\
& -\left(  \varphi_{s_{1}}^{\ast}\omega_{1}\right)  \wedge\left(  \int_{s_{2}%
}^{t}\omega_{2}\right)  +\int_{s_{2}}^{t}\left(  \varphi_{u}^{\ast}\omega
_{1}\right)  \wedge\left(  \omega_{2}\right)  _{u}^{\ast}\mathbf{d}u\\
& -\left(  -1\right)  ^{p_{1}}\left\{  \left(  \int_{s_{1}}^{t}\omega
_{1}\right)  \wedge\varphi_{t}^{\ast}\omega_{2}-\left(  \int_{s_{1}}^{s_{2}%
}\omega_{1}\right)  \wedge\varphi_{s_{2}}^{\ast}\omega_{2}-\int_{s_{2}}%
^{t}\left(  \omega_{1}\right)  _{u}^{\ast}\wedge\varphi_{u}^{\ast}\omega
_{2}\mathbf{d}u\right\} \\
& \text{[By Lemma \ref{t4.5}]}\\
& =-\int_{s_{1},s_{2}}^{t}\left(  \mathbf{d}\omega_{1}\right)  \omega
_{2}+\left(  -1\right)  ^{p_{1}}\int_{s_{1},s_{2}}^{t}\omega_{1}\left(
\mathbf{d}\omega_{2}\right) \\
& +\left\{  \int_{s_{2}}^{t}\left(  \varphi_{u}^{\ast}\omega_{1}\right)
\wedge\left(  \omega_{2}\right)  _{u}^{\ast}\mathbf{d}u+\left(  -1\right)
^{p_{1}}\int_{s_{2}}^{t}\left(  \omega_{1}\right)  _{u}^{\ast}\wedge
\varphi_{u}^{\ast}\omega_{2}\mathbf{d}u\right\} \\
& -\left(  \varphi_{s_{1}}^{\ast}\omega_{1}\right)  \wedge\left(  \int_{s_{2}%
}^{t}\omega_{2}\right)  +\left(  -1\right)  ^{p_{1}}\left(  \int_{s_{1}%
}^{s_{2}}\omega_{1}\right)  \wedge\varphi_{s_{2}}^{\ast}\omega_{2}-\left(
-1\right)  ^{p_{1}}\left(  \int_{s_{1}}^{t}\omega_{1}\right)  \wedge
\varphi_{t}^{\ast}\omega_{2}\\
& =-\int_{s_{1},s_{2}}^{t}\left(  \mathbf{d}\omega_{1}\right)  \omega
_{2}+\left(  -1\right)  ^{p_{1}}\int_{s_{1},s_{2}}^{t}\omega_{1}\left(
\mathbf{d}\omega_{2}\right)  +\left(  -1\right)  ^{p_{1}}\int_{s_{2}}%
^{t}\omega_{1}\wedge\omega_{2}\\
& -\left(  \varphi_{s_{1}}^{\ast}\omega_{1}\right)  \wedge\left(  \int_{s_{2}%
}^{t}\omega_{2}\right)  +\left(  -1\right)  ^{p_{1}}\left(  \int_{s_{1}%
}^{s_{2}}\omega_{1}\right)  \wedge\varphi_{s_{2}}^{\ast}\omega_{2}-\left(
-1\right)  ^{p_{1}}\left(  \int_{s_{1}}^{t}\omega_{1}\right)  \wedge
\varphi_{t}^{\ast}\omega_{2}%
\end{align*}

\end{proof}

\begin{lemma}
\label{t4.7}Let $s_{1},s_{2},s_{3},t\in I$. Given $\omega_{1}\in
\mathcal{A}^{p_{1}}\left(  M\right)  $, $\omega_{2}\in\mathcal{A}^{p_{2}%
}\left(  M\right)  $ and $\omega_{3}\in\mathcal{A}^{p_{3}}\left(  M\right)  $
with $p_{1}$, $p_{2}$ and $p_{3}$ being positive integers, we have
\begin{align*}
& \int_{s_{3}}^{t}\left(  \int_{s_{1}}^{u}\omega_{1}\right)  \wedge\left(
\omega_{2}\right)  _{u}^{\ast}\wedge\varphi_{u}^{\ast}\omega_{3}%
\mathbf{d}u+\int_{s_{3}}^{t}\left(  \int_{s_{1},s_{2}}^{u}\omega_{1}\omega
_{2}\right)  \wedge\left(  \iota_{u}^{\ast}\mathbf{L}_{\frac{\partial
}{\partial t}}\varphi^{\ast}\omega_{3}\right)  \mathbf{d}u\\
& =\left(  \int_{s_{1},s_{2}}^{t}\omega_{1}\omega_{2}\right)  \wedge
\varphi_{t}^{\ast}\omega_{3}-\left(  \int_{s_{1},s_{2}}^{s_{3}}\omega
_{1}\omega_{2}\right)  \wedge\varphi_{s_{3}}^{\ast}\omega_{3}%
\end{align*}

\end{lemma}

\begin{proof}
We have
\begin{align*}
& \mathbf{D}_{u}\left(  \left(  \int_{s_{1},s_{2}}^{u}\omega_{1}\omega
_{2}\right)  \wedge\varphi_{u}^{\ast}\omega_{3}\right) \\
& =\left(  \int_{s_{1}}^{u}\omega_{1}\right)  \wedge\left(  \omega_{2}\right)
_{u}^{\ast}\wedge\varphi_{u}^{\ast}\omega_{3}+\left(  \int_{s_{1},s_{2}}%
^{u}\omega_{1}\omega_{2}\right)  \wedge\left(  \iota_{u}^{\ast}\mathbf{L}%
_{\frac{\partial}{\partial t}}\varphi^{\ast}\omega_{3}\right)
\end{align*}

\end{proof}

\begin{theorem}
\label{t4.8}Let $s_{1},s_{2},s_{3},t\in I$. Given $\omega_{1}\in
\mathcal{A}^{p_{1}}\left(  M\right)  $, $\omega_{2}\in\mathcal{A}^{p_{2}%
}\left(  M\right)  $ and $\omega_{3}\in\mathcal{A}^{p_{3}}\left(  M\right)  $
with $p_{1}$, $p_{2}$ and $p_{3}$ being positive integers, we have
\begin{align*}
& \mathbf{d}\int_{s_{1},s_{2},s_{3}}^{t}\omega_{1}\omega_{2}\omega_{3}\\
& =-\int_{s_{1},s_{2},s_{3}}^{t}\left(  \mathbf{d}\omega_{1}\right)
\omega_{2}\omega_{3}+\left(  -1\right)  ^{p_{1}}\int_{s_{1},s_{2},s_{3}}%
^{t}\omega_{1}\left(  \mathbf{d}\omega_{2}\right)  \omega_{3}-\left(
-1\right)  ^{p_{1}+p_{2}}\int_{s_{1},s_{2},s_{3}}^{t}\omega_{1}\omega
_{2}\left(  \mathbf{d}\omega_{3}\right) \\
& +\left(  -1\right)  ^{p_{1}}\int_{s_{2},s_{3}}^{t}\left(  \omega_{1}%
\wedge\omega_{2}\right)  \omega_{3}-\left(  -1\right)  ^{p_{1}+p_{2}}%
\int_{s_{1},s_{3}}^{t}\omega_{1}\left(  \omega_{2}\wedge\omega_{3}\right)
-\varphi_{s_{1}}^{\ast}\omega_{1}\wedge\int_{s_{2},s_{3}}^{t}\omega_{2}%
\omega_{3}\\
& +\left(  -1\right)  ^{p_{1}}\left(  \int_{s_{1}}^{s_{2}}\omega_{1}\right)
\wedge\left(  \varphi_{s_{2}}^{\ast}\omega_{2}\right)  \wedge\left(
\int_{s_{3}}^{t}\omega_{3}\right)  -\left(  -1\right)  ^{p_{1}+p_{2}}\left(
\int_{s_{1},s_{2}}^{s_{3}}\omega_{1}\omega_{2}\right)  \wedge\varphi_{s_{3}%
}^{\ast}\omega_{3}\\
& +\left(  -1\right)  ^{p_{1}+p_{2}}\left(  \int_{s_{1},s_{2}}^{t}\omega
_{1}\omega_{2}\right)  \wedge\varphi_{t}^{\ast}\omega_{3}%
\end{align*}
In particular, we have
\begin{align*}
& \mathbf{d}\int\omega_{1}\omega_{2}\omega_{3}\\
& =-\int\left(  \mathbf{d}\omega_{1}\right)  \omega_{2}\omega_{3}+\left(
-1\right)  ^{p_{1}}\int\omega_{1}\left(  \mathbf{d}\omega_{2}\right)
\omega_{3}-\left(  -1\right)  ^{p_{1}+p_{2}}\int\omega_{1}\omega_{2}\left(
\mathbf{d}\omega_{3}\right) \\
& +\left(  -1\right)  ^{p_{1}}\int\left(  \omega_{1}\wedge\omega_{2}\right)
\omega_{3}-\left(  -1\right)  ^{p_{1}+p_{2}}\int\omega_{1}\left(  \omega
_{2}\wedge\omega_{3}\right) \\
& -\varphi_{0}^{\ast}\omega_{1}\wedge\int\omega_{2}\omega_{3}+\left(
-1\right)  ^{p_{1}+p_{2}}\left(  \int\omega_{1}\omega_{2}\right)
\wedge\varphi_{1}^{\ast}\omega_{3}%
\end{align*}

\end{theorem}

\begin{proof}
We have
\begin{align*}
& \mathbf{d}\int_{s_{1},s_{2},s_{3}}^{t}\omega_{1}\omega_{2}\omega_{3}\\
& =\mathbf{d}\int_{s_{3}}^{t}\left(  \int_{s_{1},s_{2}}^{u}\omega_{1}%
\omega_{2}\right)  \wedge\left(  \omega_{3}\right)  _{u}^{\ast}\mathbf{d}u\\
& =\int_{s_{3}}^{t}\left\{  \left(  \mathbf{d}\int_{s_{1},s_{2}}^{u}\omega
_{1}\omega_{2}\right)  \wedge\left(  \omega_{3}\right)  _{u}^{\ast}+\left(
-1\right)  ^{p_{1}+p_{2}}\left(  \int_{s_{1},s_{2}}^{u}\omega_{1}\omega
_{2}\right)  \wedge\mathbf{d}\left(  \omega_{3}\right)  _{u}^{\ast}\right\}
\mathbf{d}u\\
& \text{[By Proposition \ref{t3.4}]}\\
& =\int_{s_{3}}^{t}\left(  \mathbf{d}\int_{s_{1},s_{2}}^{u}\omega_{1}%
\omega_{2}\right)  \wedge\left(  \omega_{3}\right)  _{u}^{\ast}\mathbf{d}%
u+\left(  -1\right)  ^{p_{1}+p_{2}}\int_{s_{3}}^{t}\left(  \int_{s_{1},s_{2}%
}^{u}\omega_{1}\omega_{2}\right)  \wedge\mathbf{d}\left(  \omega_{3}\right)
_{u}^{\ast}\mathbf{d}u\\
& =\int_{s_{3}}^{t}\left(
\begin{array}
[c]{c}%
-\int_{s_{1},s_{2}}^{u}\left(  \mathbf{d}\omega_{1}\right)  \omega_{2}+\left(
-1\right)  ^{p_{1}}\int_{s_{1},s_{2}}^{u}\omega_{1}\left(  \mathbf{d}%
\omega_{2}\right)  +\left(  -1\right)  ^{p_{1}}\int_{s_{2}}^{u}\omega
_{1}\wedge\omega_{2}\\
-\varphi_{s_{1}}^{\ast}\omega_{1}\wedge\int_{s_{2}}^{u}\omega_{2}+\left(
-1\right)  ^{p_{1}}\left(  \int_{s_{1}}^{s_{2}}\omega_{1}\right)
\wedge\varphi_{s_{2}}^{\ast}\omega_{2}\\
-\left(  -1\right)  ^{p_{1}}\left(  \int_{s_{1}}^{u}\omega_{1}\right)
\wedge\varphi_{u}^{\ast}\omega_{2}%
\end{array}
\right)  \wedge\left(  \omega_{3}\right)  _{u}^{\ast}\mathbf{d}u\\
& +\left(  -1\right)  ^{p_{1}+p_{2}}\int_{s_{3}}^{t}\left(  \int_{s_{1},s_{2}%
}^{u}\omega_{1}\omega_{2}\right)  \wedge\left(  \iota_{u}^{\ast}%
\mathbf{L}_{\frac{\partial}{\partial t}}\varphi^{\ast}\omega_{3}-\left(
\mathbf{d}\omega_{3}\right)  _{u}^{\ast}\right)  \mathbf{d}u\\
& \text{[By Theorem \ref{t4.6} and Proposition \ref{t3.2}]}\\
& =-\int_{s_{1},s_{2},s_{3}}^{t}\left(  \mathbf{d}\omega_{1}\right)
\omega_{2}\omega_{3}+\left(  -1\right)  ^{p_{1}}\int_{s_{1},s_{2},s_{3}}%
^{t}\omega_{1}\left(  \mathbf{d}\omega_{2}\right)  \omega_{3}-\left(
-1\right)  ^{p_{1}+p_{2}}\int_{s_{1},s_{2},s_{3}}^{t}\omega_{1}\omega
_{2}\left(  \mathbf{d}\omega_{3}\right) \\
& +\left(  -1\right)  ^{p_{1}}\int_{s_{2},s_{3}}^{t}\left(  \omega_{1}%
\wedge\omega_{2}\right)  \omega_{3}-\varphi_{s_{1}}^{\ast}\omega_{1}\wedge
\int_{s_{2},s_{3}}^{t}\omega_{2}\omega_{3}\\
& +\left(  -1\right)  ^{p_{1}}\left(  \int_{s_{1}}^{s_{2}}\omega_{1}\right)
\wedge\varphi_{s_{2}}^{\ast}\omega_{2}\wedge\int_{s_{3}}^{t}\omega_{3}-\left(
-1\right)  ^{p_{1}}\int_{s_{3}}^{t}\left(  \int_{s_{1}}^{u}\omega_{1}\right)
\wedge\varphi_{u}^{\ast}\omega_{2}\wedge\left(  \omega_{3}\right)  _{u}^{\ast
}\mathbf{d}u\\
& +\left(  -1\right)  ^{p_{1}+p_{2}}\int_{s_{3}}^{t}\left(  \int_{s_{1},s_{2}%
}^{u}\omega_{1}\omega_{2}\right)  \wedge\left(  \iota_{u}^{\ast}%
\mathbf{L}_{\frac{\partial}{\partial t}}\varphi^{\ast}\omega_{3}\right)
\mathbf{d}u\\
& =-\int_{s_{1},s_{2},s_{3}}^{t}\left(  \mathbf{d}\omega_{1}\right)
\omega_{2}\omega_{3}+\left(  -1\right)  ^{p_{1}}\int_{s_{1},s_{2},s_{3}}%
^{t}\omega_{1}\left(  \mathbf{d}\omega_{2}\right)  \omega_{3}-\left(
-1\right)  ^{p_{1}+p_{2}}\int_{s_{1},s_{2},s_{3}}^{t}\omega_{1}\omega
_{2}\left(  \mathbf{d}\omega_{3}\right) \\
& +\left(  -1\right)  ^{p_{1}}\int_{s_{2},s_{3}}^{t}\left(  \omega_{1}%
\wedge\omega_{2}\right)  \omega_{3}-\varphi_{s_{1}}^{\ast}\omega_{1}\wedge
\int_{s_{2},s_{3}}^{t}\omega_{2}\omega_{3}\\
& +\left(  -1\right)  ^{p_{1}}\left(  \int_{s_{1}}^{s_{2}}\omega_{1}\right)
\wedge\varphi_{s_{2}}^{\ast}\omega_{2}\wedge\int_{s_{3}}^{t}\omega_{3}-\left(
-1\right)  ^{p_{1}}\int_{s_{3}}^{t}\left(  \int_{s_{1}}^{u}\omega_{1}\right)
\wedge\varphi_{u}^{\ast}\omega_{2}\wedge\left(  \omega_{3}\right)  _{u}^{\ast
}\mathbf{d}u\\
& +\left(  -1\right)  ^{p_{1}+p_{2}}\left\{
\begin{array}
[c]{c}%
-\int_{s_{3}}^{t}\left(  \int_{s_{1}}^{u}\omega_{1}\right)  \wedge\left(
\omega_{2}\right)  _{u}^{\ast}\wedge\varphi_{u}^{\ast}\omega_{3}\mathbf{d}u\\
+\left(  \int_{s_{1},s_{2}}^{t}\omega_{1}\omega_{2}\right)  \wedge\varphi
_{t}^{\ast}\omega_{3}-\left(  \int_{s_{1},s_{2}}^{s_{3}}\omega_{1}\omega
_{2}\right)  \wedge\varphi_{s_{3}}^{\ast}\omega_{3}%
\end{array}
\right\} \\
& \text{[By Lemma \ref{t4.7}]}%
\end{align*}
we keep on.
\begin{align*}
& =-\int_{s_{1},s_{2},s_{3}}^{t}\left(  \mathbf{d}\omega_{1}\right)
\omega_{2}\omega_{3}+\left(  -1\right)  ^{p_{1}}\int_{s_{1},s_{2},s_{3}}%
^{t}\omega_{1}\left(  \mathbf{d}\omega_{2}\right)  \omega_{3}-\left(
-1\right)  ^{p_{1}+p_{2}}\int_{s_{1},s_{2},s_{3}}^{t}\omega_{1}\omega
_{2}\left(  \mathbf{d}\omega_{3}\right) \\
& +\left(  -1\right)  ^{p_{1}}\int_{s_{2},s_{3}}^{t}\left(  \omega_{1}%
\wedge\omega_{2}\right)  \omega_{3}\\
& -\left(  -1\right)  ^{p_{1}+p_{2}}\left(
\begin{array}
[c]{c}%
\int_{s_{3}}^{t}\left(  \int_{s_{1}}^{u}\omega_{1}\right)  \wedge\left(
\omega_{2}\right)  _{u}^{\ast}\wedge\varphi_{u}^{\ast}\omega_{3}\mathbf{d}u\\
+\left(  -1\right)  ^{p_{2}}\int_{s_{3}}^{t}\left(  \int_{s_{1}}^{u}\omega
_{1}\right)  \wedge\varphi_{u}^{\ast}\omega_{2}\wedge\left(  \omega
_{3}\right)  _{u}^{\ast}\mathbf{d}u
\end{array}
\right) \\
& -\varphi_{s_{1}}^{\ast}\omega_{1}\wedge\int_{s_{2},s_{3}}^{t}\omega
_{2}\omega_{3}+\left(  -1\right)  ^{p_{1}}\left(  \int_{s_{1}}^{s_{2}}%
\omega_{1}\right)  \wedge\left(  \varphi_{s_{2}}^{\ast}\omega_{2}\right)
\wedge\left(  \int_{s_{3}}^{t}\omega_{3}\right) \\
& +\left(  -1\right)  ^{p_{1}+p_{2}}\left(  \int_{s_{1},s_{2}}^{t}\omega
_{1}\omega_{2}\right)  \wedge\varphi_{t}^{\ast}\omega_{3}-\left(  -1\right)
^{p_{1}+p_{2}}\left(  \int_{s_{1},s_{2}}^{s_{3}}\omega_{1}\omega_{2}\right)
\wedge\varphi_{s_{3}}^{\ast}\omega_{3}\\
& =-\int_{s_{1},s_{2},s_{3}}^{t}\left(  \mathbf{d}\omega_{1}\right)
\omega_{2}\omega_{3}+\left(  -1\right)  ^{p_{1}}\int_{s_{1},s_{2},s_{3}}%
^{t}\omega_{1}\left(  \mathbf{d}\omega_{2}\right)  \omega_{3}-\left(
-1\right)  ^{p_{1}+p_{2}}\int_{s_{1},s_{2},s_{3}}^{t}\omega_{1}\omega
_{2}\left(  \mathbf{d}\omega_{3}\right) \\
& +\left(  -1\right)  ^{p_{1}}\int_{s_{2},s_{3}}^{t}\left(  \omega_{1}%
\wedge\omega_{2}\right)  \omega_{3}-\left(  -1\right)  ^{p_{1}+p_{2}}%
\int_{s_{1},s_{3}}^{t}\omega_{1}\left(  \omega_{2}\wedge\omega_{3}\right)
-\varphi_{s_{1}}^{\ast}\omega_{1}\wedge\int_{s_{2},s_{3}}^{t}\omega_{2}%
\omega_{3}\\
& +\left(  -1\right)  ^{p_{1}}\left(  \int_{s_{1}}^{s_{2}}\omega_{1}\right)
\wedge\left(  \varphi_{s_{2}}^{\ast}\omega_{2}\right)  \wedge\left(
\int_{s_{3}}^{t}\omega_{3}\right)  +\left(  -1\right)  ^{p_{1}+p_{2}}\left(
\int_{s_{1},s_{2}}^{t}\omega_{1}\omega_{2}\right)  \wedge\varphi_{t}^{\ast
}\omega_{3}\\
& -\left(  -1\right)  ^{p_{1}+p_{2}}\left(  \int_{s_{1},s_{2}}^{s_{3}}%
\omega_{1}\omega_{2}\right)  \wedge\varphi_{s_{3}}^{\ast}\omega_{3}%
\end{align*}

\end{proof}

In general, we have

\begin{theorem}
\label{t4.9}Let $s_{1},...,s_{k},t\in I$. Given $\omega_{i}\in\mathcal{A}%
^{p_{i}}\left(  M\right)  $ with $p_{i}$ being a positive integer ($1\leq
i\leq k$), we have
\begin{align*}
& \mathbf{d}\int_{s_{1},...,s_{3}}^{t}\omega_{1}...\omega_{k}\\
& =\sum_{i=1}^{k}\left(  -1\right)  ^{i+p_{1}+...+p_{i-1}}\int_{s_{1}%
,...,s_{k}}^{t}\omega_{1}...\omega_{i-1}\left(  \mathbf{d}\omega_{i}\right)
\omega_{i+1}...\omega_{k}\\
& +\sum_{i=1}^{k-1}\left(  -1\right)  ^{i+p_{1}+...+p_{i}+1}\int
_{s_{1},...,s_{k}}^{t}\omega_{1}...\omega_{i-1}\left(  \omega_{i}\wedge
\omega_{i+1}\right)  \omega_{i+2}...\omega_{k}\\
& +\sum_{i=1}^{k}\left(  -1\right)  ^{i+p_{1}+...+p_{i-1}}\left(  \int
_{s_{1},...,s_{i-1}}^{s_{i}}\omega_{1}...\omega_{i-1}\right)  \wedge\left(
\varphi_{s_{i}}^{\ast}\omega_{i}\right)  \wedge\left(  \int_{s_{i+1}%
,...,s_{k}}^{t}\omega_{i+1}...\omega_{k}\right) \\
& +\left(  -1\right)  ^{p_{1}+...+p_{k-1}+k+1}\left(  \int_{s_{1},...,s_{k-1}%
}^{t}\omega_{1}...\omega_{k-1}\right)  \wedge\left(  \varphi_{t}^{\ast}%
\omega_{k}\right)
\end{align*}
In particular, we have
\begin{align*}
& \mathbf{d}\int\omega_{1}...\omega_{k}\\
& =\sum_{i=1}^{k}\left(  -1\right)  ^{i+p_{1}+...+p_{i-1}}\int\omega
_{1}...\omega_{i-1}\left(  \mathbf{d}\omega_{i}\right)  \omega_{i+1}%
...\omega_{k}\\
& +\sum_{i=1}^{k-1}\left(  -1\right)  ^{i+p_{1}+...+p_{i}+1}\int\omega
_{1}...\omega_{i-1}\left(  \omega_{i}\wedge\omega_{i+1}\right)  \omega
_{i+2}...\omega_{k}\\
& -\left(  \varphi_{0}^{\ast}\omega_{1}\right)  \wedge\left(  \int\omega
_{2}...\omega_{k}\right) \\
& +\left(  -1\right)  ^{p_{1}+...+p_{k-1}+k+1}\left(  \int\omega_{1}%
...\omega_{k-1}\right)  \wedge\left(  \varphi_{1}^{\ast}\omega_{k}\right)
\end{align*}

\end{theorem}

\begin{corollary}
\label{t4.9.1}Given $\omega_{i}\in\mathcal{A}^{p_{i}}\left(  M\right)  $
($0\leq i\leq k+1$)\ with $p_{i}$ being a positive integer ($1\leq i\leq k$)
and $p_{0}$\ and $p_{k+1}$\ being non-negative integers, we have
\begin{align*}
& \mathbf{d}\left(  \left(  \varphi_{0}^{\ast}\omega_{0}\right)  \wedge\left(
\int\omega_{1}...\omega_{k}\right)  \wedge\left(  \varphi_{1}^{\ast}%
\omega_{k+1}\right)  \right) \\
& =\left(  \varphi_{0}^{\ast}\left(  \mathbf{d}\omega_{0}\right)  \right)
\wedge\left(  \int\omega_{1}...\omega_{k}\right)  \wedge\left(  \varphi
_{1}^{\ast}\omega_{k+1}\right) \\
& +\sum_{i=1}^{k}\left(  -1\right)  ^{i+p_{0}+p_{1}+...+p_{i-1}}\left(
\varphi_{0}^{\ast}\omega_{0}\right)  \wedge\left(  \int\omega_{1}%
...\omega_{i-1}\left(  \mathbf{d}\omega_{i}\right)  \omega_{i+1}...\omega
_{k}\right)  \wedge\left(  \varphi_{1}^{\ast}\omega_{k+1}\right) \\
& +\sum_{i=1}^{k-1}\left(  -1\right)  ^{i+p_{0}+p_{1}+...+p_{i}+1}\left(
\varphi_{0}^{\ast}\omega_{0}\right)  \wedge\left(  \int\omega_{1}%
...\omega_{i-1}\left(  \omega_{i}\wedge\omega_{i+1}\right)  \omega
_{i+2}...\omega_{k}\right)  \wedge\left(  \varphi_{1}^{\ast}\omega
_{k+1}\right) \\
& -\left(  -1\right)  ^{p_{0}}\left(  \varphi_{0}^{\ast}\left(  \omega
_{0}\wedge\omega_{1}\right)  \right)  \wedge\left(  \int\omega_{2}%
...\omega_{k}\right)  \wedge\left(  \varphi_{1}^{\ast}\omega_{k+1}\right) \\
& +\left(  -1\right)  ^{p_{0}+p_{1}+...+p_{k-1}+k+1}\left(  \int\omega
_{1}...\omega_{k-1}\right)  \wedge\left(  \varphi_{1}^{\ast}\left(  \omega
_{k}\wedge\omega_{k+1}\right)  \right) \\
& +\left(  -1\right)  ^{p_{0}+p_{1}+...+p_{k}+k}\left(  \varphi_{0}^{\ast
}\omega_{0}\right)  \wedge\left(  \int\omega_{1}...\omega_{k}\right)
\wedge\left(  \varphi_{1}^{\ast}\mathbf{d}\omega_{k+1}\right)
\end{align*}
so that the graded submodule of the de Rham module on $\mathcal{P}M$ linearly
generated by differential forms of the form $\left(  \varphi_{0}^{\ast}%
\omega_{0}\right)  \wedge\left(  \int\omega_{1}...\omega_{k}\right)
\wedge\left(  \varphi_{1}^{\ast}\omega_{k+1}\right)  $ is closed under
exterior differentiation, constituting a subcomplex of the de Rham complex
$\mathcal{A}\left(  \mathcal{P}M\right)  $.
\end{corollary}

\begin{definition}
The subcomplex in the above corollary is called the Hochschild complex
associated with the de Rham complex $\mathcal{A}\left(  M\right)  $.
\end{definition}

We conclude this section with the following simple proposition.

\begin{proposition}
\label{t4.10}Let $s_{1},s_{2},s_{1}^{\prime},s_{2}^{\prime},t,t^{\prime}\in
I$. Let $\omega_{1}\in\mathcal{A}^{p_{1}}\left(  M\right)  $ and $\omega
_{2}\in\mathcal{A}^{p_{2}}\left(  M\right)  $. We have
\begin{align*}
\int_{s_{1},s_{2}}^{t}\omega_{1}\omega_{2}  & =\int_{s_{1}^{\prime},s_{2}}%
^{t}\omega_{1}\omega_{2}+\left(  \int_{s_{1}}^{s_{1}^{\prime}}\omega
_{1}\right)  \wedge\left(  \int_{s_{2}}^{t}\omega_{2}\right) \\
\int_{s_{1},s_{2}}^{t}\omega_{1}\omega_{2}  & =\int_{s_{1},s_{2}}%
^{s_{2}^{\prime}}\omega_{1}\omega_{2}+\int_{s_{1},s_{2}^{\prime}}^{t}%
\omega_{1}\omega_{2}\\
\int_{s_{1},s_{2}}^{t}\omega_{1}\omega_{2}  & =\int_{s_{1},s_{2}}^{t^{\prime}%
}\omega_{1}\omega_{2}+\int_{s_{1},t^{\prime}}^{t}\omega_{1}\omega_{2}%
\end{align*}

\end{proposition}

\begin{proof}
Here we deal only with the first formula. We have
\begin{align*}
& \int_{s_{1},s_{2}}^{t}\omega_{1}\omega_{2}\\
& =\int_{s_{2}}^{t}\left(  \int_{s_{1}}^{u}\omega_{1}\right)  \wedge\left(
\omega_{2}\right)  _{u}^{\ast}\mathbf{d}u\\
& =\int_{s_{2}}^{t}\left(  \int_{s_{1}^{\prime}}^{u}\omega_{1}+\int_{s_{1}%
}^{s_{1}^{\prime}}\omega_{1}\right)  \wedge\left(  \omega_{2}\right)
_{u}^{\ast}\mathbf{d}u\\
& =\int_{s_{2}}^{t}\left(  \int_{s_{1}^{\prime}}^{u}\omega_{1}\right)
\wedge\left(  \omega_{2}\right)  _{u}^{\ast}\mathbf{d}u+\int_{s_{2}}%
^{t}\left(  \int_{s_{1}}^{s_{1}^{\prime}}\omega_{1}\right)  \wedge\left(
\omega_{2}\right)  _{u}^{\ast}\mathbf{d}u\\
& =\int_{s_{1},s_{2}^{\prime}}^{t}\omega_{1}\omega_{2}+\left(  \int_{s_{1}%
}^{s_{1}^{\prime}}\omega_{1}\right)  \wedge\left(  \int_{s_{2}}^{t}\omega
_{2}\right)
\end{align*}
\begin{align*}
& \int_{s_{1},s_{2}}^{t}\omega_{1}\omega_{2}\\
& =\int_{s_{2}}^{t}\left(  \int_{s_{1}}^{u}\omega_{1}\right)  \wedge\left(
\omega_{2}\right)  _{u}^{\ast}\mathbf{d}u\\
& =\int_{s_{2}}^{s_{2}^{\prime}}\left(  \int_{s_{1}}^{u}\omega_{1}\right)
\wedge\left(  \omega_{2}\right)  _{u}^{\ast}\mathbf{d}u+\int_{s_{2}^{\prime}%
}^{t}\left(  \int_{s_{1}}^{u}\omega_{1}\right)  \wedge\left(  \omega
_{2}\right)  _{u}^{\ast}\mathbf{d}u\\
& =\int_{s_{1},s_{2}}^{s_{2}^{\prime}}\omega_{1}\omega_{2}+\int_{s_{1}%
,s_{2}^{\prime}}^{t}\omega_{1}\omega_{2}%
\end{align*}

\end{proof}

\begin{corollary}
\label{t4.10.1}
\begin{align*}
& \int_{s_{1},s_{2}}^{t}\omega_{1}\omega_{2}\\
& =\int_{s_{1}^{\prime},s_{2}}^{s_{2}^{\prime}}\omega_{1}\omega_{2}%
+\int_{s_{1}^{\prime},s_{2}^{\prime}}^{t^{\prime}}\omega_{1}\omega_{2}%
+\int_{s_{1}^{\prime},t^{\prime}}^{t}\omega_{1}\omega_{2}\\
& +\left(  \int_{s_{1}}^{s_{1}^{\prime}}\omega_{1}\right)  \wedge\left(
\int_{s_{2}}^{s_{2}^{\prime}}\omega_{2}\right)  +\left(  \int_{s_{1}}%
^{s_{1}^{\prime}}\omega_{1}\right)  \wedge\left(  \int_{s_{2}^{\prime}%
}^{t^{\prime}}\omega_{2}\right)  +\left(  \int_{s_{1}}^{s_{1}^{\prime}}%
\omega_{1}\right)  \wedge\left(  \int_{t^{\prime}}^{t}\omega_{2}\right)
\end{align*}

\end{corollary}

\begin{proof}
We have
\begin{align*}
& \int_{s_{1},s_{2}}^{t}\omega_{1}\omega_{2}\\
& =\int_{s_{1}^{\prime},s_{2}}^{t}\omega_{1}\omega_{2}+\left(  \int_{s_{1}%
}^{s_{1}^{\prime}}\omega_{1}\right)  \wedge\left(  \int_{s_{2}}^{t}\omega
_{2}\right) \\
& =\int_{s_{1}^{\prime},s_{2}}^{s_{2}^{\prime}}\omega_{1}\omega_{2}%
+\int_{s_{1}^{\prime},s_{2}^{\prime}}^{t}\omega_{1}\omega_{2}\\
& +\left(  \int_{s_{1}}^{s_{1}^{\prime}}\omega_{1}\right)  \wedge\left(
\int_{s_{2}}^{s_{2}^{\prime}}\omega_{2}\right)  +\left(  \int_{s_{1}}%
^{s_{1}^{\prime}}\omega_{1}\right)  \wedge\left(  \int_{s_{2}^{\prime}%
}^{t^{\prime}}\omega_{2}\right)  +\left(  \int_{s_{1}}^{s_{1}^{\prime}}%
\omega_{1}\right)  \wedge\left(  \int_{t^{\prime}}^{t}\omega_{2}\right) \\
& =\int_{s_{1}^{\prime},s_{2}}^{s_{2}^{\prime}}\omega_{1}\omega_{2}%
+\int_{s_{1}^{\prime},s_{2}^{\prime}}^{t^{\prime}}\omega_{1}\omega_{2}%
+\int_{s_{1}^{\prime},t^{\prime}}^{t}\omega_{1}\omega_{2}\\
& +\left(  \int_{s_{1}}^{s_{1}^{\prime}}\omega_{1}\right)  \wedge\left(
\int_{s_{2}}^{s_{2}^{\prime}}\omega_{2}\right)  +\left(  \int_{s_{1}}%
^{s_{1}^{\prime}}\omega_{1}\right)  \wedge\left(  \int_{s_{2}^{\prime}%
}^{t^{\prime}}\omega_{2}\right)  +\left(  \int_{s_{1}}^{s_{1}^{\prime}}%
\omega_{1}\right)  \wedge\left(  \int_{t^{\prime}}^{t}\omega_{2}\right)
\end{align*}

\end{proof}

\section{\label{s5}Based Path Spaces}

Let us begin by fixing our notation.

\begin{notation}
Let $x_{1},x_{2}\in M$. We denote by $\mathcal{P}_{x_{1}}M$, $\mathcal{P}%
^{x_{2}}M$ and $\mathcal{P}_{x_{1}}^{x_{2}}M$ the spaces
\begin{align*}
\mathcal{P}_{x_{1}}M  & =\left\{  \theta\in\mathcal{P}M\mid\theta\left(
0\right)  =x_{1}\right\} \\
\mathcal{P}^{x_{2}}M  & =\left\{  \theta\in\mathcal{P}M\mid\theta\left(
1\right)  =x_{2}\right\} \\
\mathcal{P}_{x_{1}}^{x_{2}}M  & =\left\{  \theta\in\mathcal{P}M\mid
\theta\left(  0\right)  =x_{1}\text{ and }\theta\left(  1\right)
=x_{2}\right\}
\end{align*}

\end{notation}

\begin{notation}
For simplicity, we use the same notation $\varphi:I\times\mathcal{P}_{x_{1}%
}M\rightarrow M$, $\varphi:I\times\mathcal{P}^{x_{2}}M\rightarrow M$ and
$\varphi:I\times\mathcal{P}_{x_{1}}^{x_{2}}M\rightarrow M$ for the
restrictions of $\varphi:I\times\mathcal{P}M\rightarrow M$.
\end{notation}

By replacing $\mathcal{P}M$ by $\mathcal{P}_{x_{0},x_{1}}M$ throughout in the
previous two sections, we have corresponding but simpler results. In
particular, we have

\begin{theorem}
\label{t5.1}Let $x_{1}\in M$. Given $\omega_{i}\in\mathcal{A}^{p_{i}}\left(
M\right)  $ ($1\leq i\leq k+1$)\ with $p_{i}$ being a positive integer ($1\leq
i\leq k$) and $p_{k+1}$\ being a non-negative integer, we have
\begin{align*}
& \mathbf{d}\left(  \left(  \int\omega_{1}...\omega_{k}\right)  \wedge\left(
\varphi_{1}^{\ast}\omega_{k+1}\right)  \right) \\
& =\left(  \int\omega_{1}...\omega_{k}\right)  \wedge\left(  \varphi_{1}%
^{\ast}\omega_{k+1}\right) \\
& +\sum_{i=1}^{k}\left(  -1\right)  ^{i+p_{1}+...+p_{i-1}}\left(  \int
\omega_{1}...\omega_{i-1}\left(  \mathbf{d}\omega_{i}\right)  \omega
_{i+1}...\omega_{k}\right)  \wedge\left(  \varphi_{1}^{\ast}\omega
_{k+1}\right) \\
& +\sum_{i=1}^{k-1}\left(  -1\right)  ^{i+p_{1}+...+p_{i}+1}\left(  \int
\omega_{1}...\omega_{i-1}\left(  \omega_{i}\wedge\omega_{i+1}\right)
\omega_{i+2}...\omega_{k}\right)  \wedge\left(  \varphi_{1}^{\ast}\omega
_{k+1}\right) \\
& +\left(  -1\right)  ^{p_{1}+...+p_{k-1}+k+1}\left(  \int\omega_{1}%
...\omega_{k-1}\right)  \wedge\left(  \varphi_{1}^{\ast}\left(  \omega
_{k}\wedge\omega_{k+1}\right)  \right) \\
& +\left(  -1\right)  ^{p_{1}+...+p_{k}+k}\left(  \int\omega_{1}...\omega
_{k}\right)  \wedge\left(  \varphi_{1}^{\ast}\mathbf{d}\omega_{k+1}\right)
\end{align*}
so that the graded submodule of the de Rham module on $\mathcal{P}_{x_{1}}M$
linearly generated by differential forms of the form $\left(  \int\omega
_{1}...\omega_{k}\right)  \wedge\left(  \varphi_{1}^{\ast}\omega_{k+1}\right)
$ is closed under exterior differentiation, constituting a subcomplex of the
de Rham complex $\mathcal{A}\left(  \mathcal{P}_{x_{1}}M\right)  $.
\end{theorem}

\begin{theorem}
\label{t5.2}Let $x_{2}\in M$. Given $\omega_{i}\in\mathcal{A}^{p_{i}}\left(
M\right)  $ ($0\leq i\leq k$)\ with $p_{i}$ being a positive integer ($1\leq
i\leq k$) and $p_{0}$\ being a non-negative integer, we have
\begin{align*}
& \mathbf{d}\left(  \left(  \varphi_{0}^{\ast}\omega_{0}\right)  \wedge\left(
\int\omega_{1}...\omega_{k}\right)  \right) \\
& =\left(  \varphi_{0}^{\ast}\left(  \mathbf{d}\omega_{0}\right)  \right)
\wedge\left(  \int\omega_{1}...\omega_{k}\right) \\
& +\sum_{i=1}^{k}\left(  -1\right)  ^{i+p_{0}+p_{1}+...+p_{i-1}}\left(
\varphi_{0}^{\ast}\omega_{0}\right)  \wedge\left(  \int\omega_{1}%
...\omega_{i-1}\left(  \mathbf{d}\omega_{i}\right)  \omega_{i+1}...\omega
_{k}\right) \\
& +\sum_{i=1}^{k-1}\left(  -1\right)  ^{i+p_{0}+p_{1}+...+p_{i}+1}\left(
\varphi_{0}^{\ast}\omega_{0}\right)  \wedge\left(  \int\omega_{1}%
...\omega_{i-1}\left(  \omega_{i}\wedge\omega_{i+1}\right)  \omega
_{i+2}...\omega_{k}\right) \\
& -\left(  -1\right)  ^{p_{0}}\left(  \varphi_{0}^{\ast}\left(  \omega
_{0}\wedge\omega_{1}\right)  \right)  \wedge\left(  \int\omega_{2}%
...\omega_{k}\right)
\end{align*}
so that the graded submodule of the de Rham module on $\mathcal{P}^{x_{2}}M$
linearly generated by differential forms of the form $\left(  \varphi
_{0}^{\ast}\omega_{0}\right)  \wedge\left(  \int\omega_{1}...\omega
_{k}\right)  $ is closed under exterior differentiation, constituting a
subcomplex of the de Rham complex $\mathcal{A}\left(  \mathcal{P}^{x_{2}%
}M\right)  $.
\end{theorem}

\begin{theorem}
\label{t5.3}Let $x_{1},x_{2}\in M$. Given $\omega_{i}\in\mathcal{A}^{p_{i}%
}\left(  M\right)  $ ($1\leq i\leq k$)\ with $p_{i}$ being a positive integer
($1\leq i\leq k$), we have
\begin{align*}
& \mathbf{d}\int\omega_{1}...\omega_{k}\\
& =\sum_{i=1}^{k}\left(  -1\right)  ^{p_{1}+...+p_{i-1}+i}\int\omega
_{1}...\omega_{i-1}\left(  \mathbf{d}\omega_{i}\right)  \omega_{i+1}%
...\omega_{k}\\
& +\sum_{i=1}^{k-1}\left(  -1\right)  ^{p_{1}+...+p_{i}+i+1}\int\omega
_{1}...\omega_{i-1}\left(  \omega_{i}\wedge\omega_{i+1}\right)  \omega
_{i+2}...\omega_{k}%
\end{align*}
so that the graded submodule of the de Rham module on $\mathcal{P}_{x_{1}%
}^{x_{2}}M$ linearly generated by differential forms of the form $\int
\omega_{1}...\omega_{k}$ is closed under exterior differentiation,
constituting a subcomplex of the de Rham complex $\mathcal{A}\left(
\mathcal{P}_{x_{1}}^{x_{2}}M\right)  $.
\end{theorem}

\begin{definition}
The subcomplex in the above theorem is called the bar complex associated with
the de Rham complex $\mathcal{A}\left(  M\right)  $.
\end{definition}

\end{document}